\documentclass[11pt,leqno]{amsart} %Sep2.2018, zhou
\usepackage{bbm}
\usepackage{amsmath,amssymb}

\usepackage{mathrsfs}
\usepackage{amssymb}
\usepackage{amsmath}
\usepackage{amsthm}
\usepackage{amsfonts}
\usepackage{color}
\usepackage{graphicx}
\usepackage[active]{srcltx}
\usepackage{tikz}
\usepackage[cp1252]{inputenc}

\usepackage{mathrsfs}
\usepackage{graphicx}

\usepackage[active]{srcltx}

\allowdisplaybreaks

\usepackage{titletoc}

\def\rr{{\mathbb R}}

\def\nn{{\mathbb N}}
\def\hh{{\mathbb H}}

\def\fz{\infty}
\def\az{\alpha}

\def\loc{{\mathop\mathrm{\,loc\,}}}

\def\lz{\lambda}

\def\dzz{\Delta_0}
\def\ez{\epsilon}

\def\bz{\beta}

\def\bint{{\ifinner\rlap{\bf\kern.35em--}
\int\else\rlap{\bf\kern.45em--}\int\fi}\ignorespaces}

\def\bbint{{\ifinner\rlap{\bf\kern.35em--}
\hspace{0.078cm}\int\else\rlap{\bf\kern.45em--}\int\fi}\ignorespaces}

\def\hn{{{\mathbb H}^n}}

\newtheorem{thm}{Theorem}[section]
\newtheorem{lem}[thm]{Lemma}%[section]     %@@!!@@!!
\newtheorem{rem}[thm]{Remark}%[section]    %@@!!@@!!
\newtheorem{cor}[thm]{Corollary}%[section]    %@@!!@@!!
\numberwithin{equation}{section}

\begin{document}

\arraycolsep=1pt
\allowdisplaybreaks

\title[$HW^{2,2}_\loc$-regularity for
 $p$-harmonic functions in  Heisenberg groups]
{$HW^{2,2}_\loc$-regularity for  $p$-harmonic functions in  Heisenberg groups}

\author{Jiayin Liu, Fa Peng  and Yuan Zhou}

          \address{ Department of Mathematics, Beihang University, Beijing 100191, P.R. China}
                    \email{ljyghy@buaa.edu.cn }

                              \address{ Department of Mathematics, Beihang University, Beijing 100191, P.R. China}
                    \email{pengfa@buaa.edu.cn }

\address{ Department of Mathematics, Beijing Normal Univeristy, Beijing 100875, P.R. China}
                    \email{yuan.zhou@bnu.edu.cn}

\thanks{  }

\date{\today }

 \begin{abstract}{
Let
$
  1<p\le4$   when $n=1$ and $1<p< 3+\frac{1}{n-1} $ when $n\ge2$.
We obtain the second-order
    horizontal Sobolev $HW^{2,2}_{\rm loc} $-regularity of $p$-harmonic functions in the Heisenberg group $\mathbb H^n$.
 This improves the known range of $p$  obtained by Domokos and Manfredi in 2005.}
\end{abstract}

\maketitle

\section{Introduction}

For $n\ge1$, denote by $\hh^n$ the $n$-th Heisenberg group,
   and by $\Omega\subset\hh^n$  any domain (open connected open set).
   Let $1<p<\fz$.
   A function $u:\Omega\to\rr$ is called $p$-harmonic (usually called harmonic when $p=2$) in $\Omega$ if
   $u\in    HW_\loc ^{1,p}(\Omega )$ is a weak solution to the $p$-subLaplace equation
\begin{equation}\label{plap}
  \Delta _{0,p}u:=\sum_{i=1}^{2n}X_i(|Xu|^{p-2}X_iu) =0  \quad in \ \Omega,
\end{equation}
that is,
$$\int_\Omega  |Xu|^{p-2}\langle X u, X \phi\rangle\,dx=0\quad\forall \phi\in C^\fz_c(\Omega).$$
Here $HW_\loc ^{1,p}(\Omega)$ is  the collection of functions $u\in L^p_\loc(\Omega)$
with their distributional horizontal derivatives $Xu=(X_1u,\cdots,X_{2n}u)\in L^p_\loc(\Omega,\rr^{2n})$.

In the linear case $p=2$,
$2$-harmonic functions are  exactly harmonic functions in $\Omega\subset \hh^n$,
and their  $C^\fz$-regularity    follows from  a result by H\"{o}rmander \cite{ho}.
In the quasilinear case $p\ne2$, the study of  regularity of $p$-harmonic functions in $\hh^n$
attracted a lot of attention  in past decades.
In particular,   their  H\"older regularity   was established by  Lu \cite{lu} and Capogna \cite{c97}. Recently, their  Lipshictz regularity and also the H\"older regularity of their horizontal gradients   were proved    by Domokos and Manfredi \cite{dm2},
 Manfredi and Mingione \cite{mm},
   Mingione, Zatorska-Goldstein and Zhong  \cite{mzz2},
   Zhong \cite{zx2}, and Mukherjee and Zhong \cite{mzh}. %Here    $\alpha\in(0,1)$ is some constant depending on $n$ and $p$.

      On the other hand, if $p>2$, Capogna \cite{c97} proved the $C^\fz$-regularity of  $p$-harmonic functions $u$  under an additional assumption that $|Xu|$  is strictly bounded from above and also away from $0$. In general, $|Xu|$ may vanish in some set and
      $u$ is also not necessarily smooth. In 2005, Domokos and Manfredi  \cite{dm1} established an interesting second order differentiability:    if
 \begin{equation}\label{rangep0}\frac{\sqrt{17}-1}{2}<p<\frac{5+\sqrt{5}}{2}\  \mbox{when $n=1$ and } 2<p<  2+\frac{n+n\sqrt{4n^2+4n-3}}{2n^2+2n-2}\  \mbox{when $n\ge 2$},
 \end{equation}
      then   any $p$-harmonic function  $u$ in $\Omega\subset \hh^n$ belongs  to
      $HW^{2,2}_\loc(\Omega) $, that is,
      the second order distributional horizontal derivative $$XXu=(X_iX_ju)_{ 2n\times 2n}\in L^2_\loc(\Omega,\rr^{2n\times2n}).$$
  The restriction of $p$ in the  range \eqref{rangep0}  comes from their approach, which is based on
  some subelliptic   Cordes  estimate  built up therein.
  For other $p$, it is then open  to prove or disprove the $HW^{2,2}_\loc $-regularity of $p$-harmonic functions in Heisenberg groups $\mathbb H^n$.

The main aim of this paper is to establish the
  $HW^{2,2}_\loc $-regularity of $p$-harmonic functions with a quantitative upper bound for
    \begin{equation}\label{rangep2}
   \mbox{ $2<p\le4$ when $n=1$ and} \
  2<p< 3+\frac{1}{n-1} \ \mbox{ when $n\ge 2$.}
  \end{equation}
Note that,   for   $1<p<2$,
the $HW^{2,2}_\loc $-regularity with a quantitative upper bound can be derived directly  from
  Zhong \cite{zx2} (see also \cite{mzh}).
Below, we summarize   our result and, for the reader's convenience, the consequence of
  Zhong of \cite{zx2}.

  \begin{thm}\label{thm1}
Let
  \begin{equation}\label{rangep}
   \mbox{ $1<p\le4$ when $n=1$ and} \
  1<p< 3+\frac{1}{n-1} \ \mbox{ when $n\ge 2$.}
  \end{equation}
If  $u$ is a  $p$-harmonic function in a domain $\Omega\subset\hh^n$,  then
  $u\in HW^{2,2}_\loc(\Omega)$. Moreover,
    if  $1<p \le 2$, then
 \begin{align} \label{bdd0}
   \int_\Omega  |XXu |^2 & \phi^2 \,dx  \le   C(n,p) K_\phi   \left[\int_{\rm spt\,(\phi)}  |Xu| ^{ {2-p} }  \,dx  \int_{\rm spt\,(\phi)}   |Xu| ^{ {p+2} }\,dx\right]^{1/2}\quad\forall \phi\in C^\fz_c(\Omega).
    \end{align}
    If $p>2$ satisfies \eqref{rangep}, then
 \begin{align}\label{yy30}
    \int_\Omega |XX u |^2 \phi^2 \, dx
  &  \le C(n,p)\int_{\rm spt(\phi)} |u |^2|XX \phi |^2  \,dx+  C(n,p)K _\phi \left[ \int_{\rm spt(\phi)} |Xu |^{ {4-p} }   \,dx\right]^{1/2}\nonumber\\
     &\quad\quad\times
   \left[ \int_{\rm spt(\phi)} |Xu | ^{ {p+2} } \,dx  \int_{\rm spt(\phi)}  |Xu | ^{ {p-2} }  \,dx\right ]^{1/4}\quad\forall \phi\in C^\fz_c(\Omega).
\end{align}\end{thm}
In this paper, $C(a,b,\cdots)$ always denotes a positive constant depending only on parameters
$a, b$ ...,  but its value may vary from line to line. For any $\phi\in C_c^\fz(\Omega)$ we write
${\rm spt}(\phi)$ as the support of $\phi$, and
 $$K_\phi=\|X\phi\|^2_{L^\fz(\Omega)} + \|\phi T\phi\|_{L^\fz(\Omega)}.$$

Note that, a  direct calculation shows that,  for each $n\ge1$, the range \eqref{rangep}
is strictly larger than \eqref{rangep0}. Indeed,
when $n=1$, one has  $$ \frac{5+\sqrt{5}}{2}<\frac{5+2.3}2=3.65<4. $$
When $n\ge2$, then
 \begin{align*}
 [3+\frac1{n-1}]  - [2+\frac{n+n\sqrt{4n^2+4n-3}}{2n^2+2n-2}] & = \frac{n(2n^2+2n-2) - [n+n\sqrt{4n^2+4n-3}](n-1)}{(n-1)(2n^2+2n-2)} \notag \\
   & = n\frac{(2n+ 4 +\frac2{n-1}) - [1+\sqrt{4n^2+4n-3}]}{2n^2+2n-2} \notag \\
   & > n\frac{2n+ 3 - \sqrt{(2n+3)^2-8n-12}}{2n^2+2n-2}  \notag \\
   & >0.
 \end{align*}

Theorem 1.1 improves the corresponding result of  Domokos and Manfredi  \cite{dm1}. Moreover, comparing with the Euclidean case, we have the following remark.
\begin{rem}\rm

(i)
When $n=1$, for any $p\in(1,\fz)$, $p$-harmonic functions in $\rr^{2 }$ belong to the Sobolev space $W^{2,2}_\loc$; see \cite{im89,mw88,dpzz}.
But in $\hh^1$, due to the possible degeneracy of $|Xu|$,  the restriction $p\le 4$ is needed to bound the right hand side of \eqref{yy30}.
It is not clear whether    the assumption $1<p\le 4$   in  Theorem 1.1 can be relaxed to the whole $p\in(1,\fz)$ or not.

(ii)
When $n\ge2$, the range  $1<p<3+\frac1{n-1}$ in Theorem \ref{thm1} is exactly the range of $p$ so that
 $p$-harmonic functions in $\rr^{2n}$ belong to the  Sobolev space $W^{2,2}_\loc$; see  \cite{mw88,dpzz}.
 For $p\ge 3+\frac1{n-1}$, it remains open to get  either  the $W^{2,2}_\loc$-regularity of  $p$-harmonic functions in $\rr^{2n}$ or the $HW^{2,2}_\loc$-regularity of  $p$-harmonic functions in $\hh^{ n}$.
\end{rem}

%\begin{rem}\rm
%The range of $p$ in Theorem \ref{thm1} improves the above range of $p$ obtained in Domokos and Manfredi  \cite{dm1}.
%Indeed, when $n=1$, we have
%$$1<\frac{ \sqrt{17}-1}{2}\quad\mbox{and}\quad \frac{5+\sqrt{5}}{2}<4$$
%and when $n\ge2$, by a direct calculation one has
%$$2+\frac{2n+n\sqrt{4n^2+4n-3}}{2n^2+2n-2}<3+\frac{1}{n-1}.$$
%\end{rem}

Finally, we sketch the idea of the proof of Theorem 1.1. Note that
in the case $1<p\le 2$, Theorem 1.1  follows from   Zhong \cite{zx2}; for the readers' convenience we will give some details.
But for $p>2$  satisfying \eqref{rangep}, we can not
get Theorem 1.1 from  either \cite{zx2}  or \cite{dm1}.
We need to use some ideas from \cite{dpzz}, and
also some lemmas from \cite{zx2} and \cite{dm1}. Below we clarify this.

Let $u \in HW^{1,p}_\loc(\Omega)$ be any  $p$-harmonic function  in $\Omega\subset\hh^n$.
Given any smooth   domain $U\Subset\Omega$, for
$p\in(1,\fz)$ and
 $\ez\in(0,1]$ we let  $u^\ez\in HW^{1,p}(U)$ be a weak solution to the regularized equation
 \begin{equation}\label{aplap}
\sum_{i=1}^{2n}X_i[(\ez+|Xv|^2)^{\frac{p-2}{2}}X_iv]=0  \quad \mbox{in $U$;\quad $v-u\in HW^{1,p}_0(U)$ }.
 \end{equation}
 For the existence, uniqueness and $C^\fz(U)$-regularity  of $u^\ez$   we refer to \cite{dr,c97} and the references therein.
  By Zhong \cite{zx2},  $Xu^\ez\in L^\fz_\loc(U)$ uniformly in $\ez\in(0,1) $ and
  $u^\ez\to u$ in $C^0(U)$ as $\ez\to0$.  Then a standard approximation argument
   allows us to conclude Theorem \ref{thm1} from the following result; see Section \ref{main} for   details.

 \begin{thm} \label{ez}
 Let $p$ be as in \eqref{rangep}.
For each $\ez\in(0,1)$ let $u^\ez\in HW^{1,2}(U)$ be the weak solution to the equation \eqref{aplap} in a smooth domain $ U\Subset\Omega$.
Then  for any $\phi\in C^\fz_c(U)$,
if $1<p\le2$, we have
    \begin{align} \label{bdd}
   \int_U  |XXu^\ez|^2 & \phi^2 \,dx  \le   C(n,p) K _\phi  \left[\int_{\rm spt\,(\phi)}  (\ez+ |Xu^\ez|^2)^{\frac{2-p}{2}} \,dx \int_{\rm spt\,(\phi)}  (\ez+|Xu^\ez|^2)^{\frac{2+p}{2}}\,dx\right]^{1/2}.
    \end{align}
    If $p>2$ satisfies \eqref{rangep}, we have
    \begin{align}\label{yy3}
    \int_U |XX u^\ez|^2 \phi^2 \, dx
  &  \le C(n,p)\int_{\rm spt\,(\phi)}   |u^\ez|^2|XX \phi |^2  \,dx+  C(n,p)K _\phi  \left[  \int_{\rm spt\,(\phi)}  |Xu^\ez|^{ {4-p} }   \,dx\right]^{1/2}\nonumber\\
     &\quad\quad\times
  \left[  \int_{\rm spt\,(\phi)}  (\ez+|Xu^\ez|^2)^{\frac{p+2}{2}} \,dx \int_{\rm spt\,(\phi)}   (\ez+|Xu^\ez|^2)^{\frac{p-2}{2}}  \,dx\right ]^{1/4}.
\end{align}
Thus,
$u^\ez\in HW^{2,2}_\loc(U)$ uniformly in $\ez \in(0,1]$.
\end{thm}

 To get Theorem \ref{ez} it suffices to prove \eqref{bdd} and \eqref{yy3}.
In the case $1<p<2$, inequality \eqref{bdd} follows from \cite{zx2}, that is, Lemmas 2.3-2.5 below. Our main contribution is to prove \eqref{yy3} in the case $p>2$.
%If $p>2$  satisfies  \eqref{rangep}, we prove  \eqref{yy3} as below.
Write $\Delta_0 v=\Delta_{0,2}v$  as the 2-subLaplacian and denote by
$D^2_0v$   the symmetrization of $XXv$, that is,
  $$D_0^2v= \left (\frac12[ X_iX_{j}v+X_{j}X_iv]
             \right )_{2n\times 2n}.$$
             First, from a fundamental inequality in \cite{dpzz} (see Lemma 4.2 below)
              we deduce  the following pointwise estimate in Lemma \ref{keylem1}:
             \begin{align}\label{pi-lpap0}
\left [\frac{n}{(p-2)^2}+\frac{1}{p-2}-(n-1)\right] (\Delta_0 u^\ez)^2 \le \frac{2n-1}{2} \left[|D_0^2u^\ez|^2-(\Delta_0 u^\ez)^2 \right]\quad \mbox{in $U$}.
\end{align}
If $p=2$, the above inequality is understood as $0 \le \frac{2n-1}{2} |D_0^2u^\ez|^2$.

Next, by using some Caccioppoli inequalities in \cite{zx2} (see Lemmas 2.3-2.5 below) to handle the
  non-commutativity of $X_i$ and $X_{i+n}$, in Lemma \ref{keylem2} we  bound the  integration of the right hand side of  \eqref{pi-lpap0} via the summation of $\eta\int_U|XXu^\ez|^2\phi^2\,dx$
  and the right hand side of \eqref{yy3}.
Finally, note that \eqref{rangep} implies that the coefficient in the left hand of \eqref{pi-lpap0} is positive.
Since  \cite[Lemma 1.1]{dm1} gives
  $$\|XXv\|_{L^2(U)}\le \sqrt{1+\frac2n}  \|\Delta_0 v\|_{L^2(U)} \mbox{ for any $v\in HW^{2,2}_0(U)$}$$,
  we are able to deduce \eqref{yy3} from Lemmas \ref{keylem1} and \ref{keylem2}.  See Section 5 for details of the proof.

\section{Preliminaries}

Let $n\ge1$ be an integer. The $n$-th Heisenberg group $\mathbb{H}^{n}$ is given by a Lie group which has a background manifold $\mathbb{R}^{2n} \times \mathbb{R}$, and whose Lie algebra has a step two stratification $\mathfrak{h}^{n}$ = $V_1 \oplus V_2$, where $V_1$ has dimension $2n$, $V_2$ has dimension 1, and $[V_1, V_1]=V_2$, $[V_1, V_2]=0$ and $[V_2, V_2]=0$.
To be  precise, the group multiplication on $\hh^n$  is given by
$$(x,t) \cdot (x',t') = (x_1 + x_1' , \cdots,x_{2n} + x_{2n}', t+t'-\frac{1}{2}\sum_{i=1}^n[x_ix_{i+n}'-x_{i+n}x'_i]) $$
for all $(x,t) = (x_1, \cdots,x_{2n}, t) $ and
$(x',t') = (x'_1, \cdots,x'_{2n}, t')$ in $\rr^{2n}\times\rr$.
Associated to this group operation,  the canonical basis of the tangent space is given by  the left invariant vector fields translated from the identity, that is,
    \begin{align*}
        X_i &= \frac{\partial}{\partial x_i} - \frac{x_{i+n}}{2}\frac{\partial}{\partial t},  \qquad i=1,\cdots ,n,   \\
        X_{i+n} &= \frac{\partial}{\partial x_{i+n}} + \frac{x_{i}}{2}\frac{\partial}{\partial t}, \qquad i=1,\cdots ,n, \\
        T&= \frac{\partial}{\partial t}.
    \end{align*}
  Write  $V_1=\,{\rm span\,}\{X_1,\cdots,X_{2n}\}$ and $V_2=\,{\rm span\,} T$, and notice that
$$[X_i,X_{i+n}]=-[X_{i+n},X_i]=X_{i}X_{i+n}-X_{i+n}X_{i}=T \ \forall i=1,\cdots,n,$$
$$X_iT=TX_i, \quad\mbox{$X_{i}X_{j}=X_{j}X_{i}$, $\forall i,j=1,\cdots,2n,|i-j|\neq n$.}$$
We refer to \cite{cdpj} for more background for Heisenberg groups.

Let  $\Omega \subset \hn$ be any domain (open connected subset).
For any $v\in C^1(\Omega)$,
denote by $Xv := (X_1v, \cdots, X_{2n}v)$ the horizontal derivative;
for any $v\in C^2(\Omega)$,  denote by $XXv=(X_iX_jv)_{ 2n\times 2n}$  the second order horizontal derivative.
For $v\in L^1_\loc(\Omega)$, both of $Xv$ and $XXu$ are explained in
  distributional sense. If $Xv$ and $XXv$ are given by some vector-valued functions, their lengths are written respectively as
  $$|Xu|=(\sum_{i=1}^{2n}|X_iu|^2)^{1/2}\quad \mbox{and}\quad |XXu|=(\sum_{i,j=1}^{2n}|X_iX_ju|^2)^{1/2}.$$

For $1<p<\fz$, the horizontal Sobolev space  $HW^{1,p}(\Omega)$   is
the collection of all functions $v\in L^p (\Omega)$ with $Xv\in L^p(\Omega;\rr^{2n})$, and equipped with the norm
$$  \|v\|_{HW^{1,p}(\Omega)} = \left(\|v\|_{L^p(\Omega)}^p+ \||Xu|\|_{L^p(\Omega)}^p \right)^{1/p} .$$
For any $m\ge2$, the $m$-order horizontal Sobolev space
  $HW  ^{m,p}(\Omega, \rr)$ is the collection of all functions $u$ with $Xu\in HW ^{m-1,p}(\Omega)$, and its norm is defined in a similar way.
  For any $m\ge1$ and $p>1$,  we  write
$HW^{m,p}_\loc (\Omega)$ as the collection of all functions $u:\Omega\to\rr$ so that
$u\in HW^{m,p}(U)$ for all $U\Subset \Omega$.
We also let $HW^{m,p}_0 (\Omega)$ be the completion of $C_c^\fz(\Omega)$ under the $\|\cdot \|_{HW  ^{m,p}(\Omega )}$-norm.

%We also write $XXv=(X_iX_j v)_{2n \times 2n}$ as
%the second order horizontal derivatives
%of  $v\in C^2(\Omega)$  or, in distributional sense, of  $v\in L^1(\Omega)$.
The following  estimate  for $v\in HW^{2,2}_0 (\Omega)$ was
 obtained by Domokos-Manfredi \cite{dm1}.
 Note that $v\in HW^{2,2}_\loc(\Omega)$ implies  $XXv\in L^2_\loc(\Omega,\rr^{2n\times 2n})$.

\begin{lem} \label{lem31}
  For any $v$ $\in$ $HW_0^{2,2}(\Omega)$, we have
 % \begin{equation}\label{tu}
%   \|Tv\|_{L^2(\Omega)} \le \frac1n \|\Delta_0 v\|_{L^2(\Omega)},
%  \end{equation}
%  and
  \begin{equation}\label{xu}
    \||XXv|\|_{L^2(\Omega)} \le \sqrt{1+\frac2n}\|\Delta_0 v\|_{L^2(\Omega)}.
  \end{equation}
    The constants above are  sharp when $\Omega = \mathbb{H}^n$.
\end{lem}

As a consequence of  Lemma \ref{lem31}, we immediately have
\begin{cor}\label{cor2.2}
For any $v \in HW^{2,2}_\loc(\Omega)$ and $\phi \in C^\fz_c(\Omega)$,  we have
\begin{align}\label{al02}
 \int_\Omega \left[|XXv|\phi\right]^2 \,dx
    \le 3\int_\Omega (\dzz v)^2\phi^2 \,dx + C(n ) \int_\Omega \left[|Xv|^2|X\phi|^2+|v|^2|XX\phi|^2 \right] \,dx.
\end{align}
\end{cor}

\begin{proof}
Without loss of generality we may assume $v\in C^\fz(\Omega)$. Then $v\phi \in C_c^\fz(\Omega)$.
Applying \eqref{xu} to $v\phi$, we obtain
  \begin{equation*}
   \int_\Omega |XX(v\phi)|^2\,dx   \le (1+\frac2n) \int_\Omega |\Delta_0( v\phi)|^2 \,dx.
  \end{equation*}
Note that, for any $\eta>0$,
\begin{align*}|XX(v\phi)|^2&= |(XXv)\phi+  Xv \otimes X\phi+ X\phi \otimes Xv +vXX\phi|^2\\
&\ge (1-\eta)
|XXv|^2\phi^2  -C(n,\eta)\left[|Xv|^2|X\phi|^2+|v|^2|XX\phi|^2\right],
\end{align*}
and
\begin{align}\label{exx.1}|\Delta_0(v\phi)|^2&= [\Delta_0v\phi+2\langle Xv,X\phi\rangle+v\Delta_0\phi]^2\nonumber\\
&\le (1+\eta)
(\Delta_0v)^2\phi^2 +C(n,\eta)\left[|Xv|^2|X\phi|^2+|v|^2|XX\phi|^2\right].
\end{align}
Thus,
\begin{align*}
  (1-\eta)\int_U |XXv|^2\phi^2 \,dx
   &\le (1+\frac{2}{n})(1+\eta)\int_U  (\Delta_0v)^2\phi^2 \,dx \\
   &\quad + C(n,\eta) \int_U\left[|Xv|^2|X\phi|^2+|v|^2|XX\phi|^2\right]\,dx,
\end{align*}
which together with a suitable choice of $\eta$ gives \eqref{al02}.
\end{proof}

%Finally,
 The following three lemmas were established by  Zhong
 \cite{zx2}; see  Lemma 3.3, Lemma 3.4, and Corollary 3.2 therein.
 %
% The bound of the mixed derivatives $XTu^\ez$ in Lemma \ref{xt}  was established by  \cite[Lemma 3.3]{zx2}   where we substitute $f(z)$ therein by $(\ez+|z|^2)^{\frac{p-2}{2}}$.
%The bound of $XXu^\ez$ and the vertical derivatives $Tu^\ez$ in Lemma \ref{xxu} and \ref{tu4} is a special case
% \cite[Corollary 3.2, Lemma 3.4]{zx2}   with $f=(\ez+|z|^2)^{\frac{p-2}{2}}$ therein.

\begin{lem} \label{xt}
For $1<p<\fz$ and
 $0<\ez<1$, let  $u^\ez\in C^\fz(U)$ be a solution to  \eqref{aplap}.
Let $\bz\ge0$. For  any $\phi \in C^\fz_c(U)$, we have
  \begin{equation}\label{xtu}
  \int_U (\ez+|Xu^\ez|^2)^{\frac{p-2}{2}}|Tu^\ez|^\bz|XTu^\ez|^2 \phi^2 \,dx \le C(n,p,\bz) \int_U (\ez+|Xu^\ez|^2)^{\frac{p-2}{2}}|Tu^\ez|^{\bz+2} |X\phi|^2  \,dx.
     \end{equation}
\end{lem}

\begin{lem} \label{xxu}
For $1<p<\fz$ and
 $0<\ez<1$, let  $u^\ez\in C^\fz(U)$ be a solution to  \eqref{aplap}.   Let $\bz\ge0$.
For  any $\phi \in C^\fz_c(U)$, we have
  \begin{align}
  \int_U (\ez+|Xu^\ez|^2)^{\frac{p-2+\bz}2}|XXu^\ez|^2 \phi^2 \,dx  \le & C (n,p)  \int_U  (\ez+|Xu^\ez|^2)^{\frac{p+\bz}2}[|X\phi|^2+|\phi T\phi|] \,dx \nonumber\\
   &+C(n,p,\bz)\int_U (\ez+|Xu^\ez|^2)^{\frac{p-2+\bz}2}|Tu^\ez|^2\phi^2 \,dx.
     \end{align}
\end{lem}

\begin{lem} \label{tu4}
For $1<p<\fz$ and
 $0<\ez<1$, let  $u^\ez\in C^\fz(U)$ be a solution to  \eqref{aplap}.   Let $\bz\ge2$.
For  any $0\le \phi \in C^\fz_c(U)$, we have
  \begin{equation}
  \int_U (\ez+|X u^\ez|^2)^{\frac{p-2}{2}}|Tu^\ez|^{\bz+2} \phi^{\bz+2} \,dx  \le  C(n,p,\bz)K_\phi ^{\frac{\bz+2}2} \int_{\rm spt\,(\phi)}  (\ez+|Xu^\ez|^2)^{\frac{p+\bz}{2}} \,dx.
     \end{equation}
\end{lem}

Zhong \cite{zx2}   further deduced the following uniform gradient estimate and also convergence.
 \begin{thm} \label{unif}
For $1<p<\fz$ and
 $0\le \ez<1$, let $u^\ez\in HW^{1,p}(U)$ be a solution to  \eqref{aplap}, where we write $u^0=u$.
Then $Xu^\ez \in L^\fz_{\loc}(U,\rr^{2n})$ uniformly in $\ez\in[0,1)$  and,  for any ball $B_{2r} \subset U$,
\begin{equation}\label{unif2}
  \|\,|Xu^\ez|\,\|_{ L^\fz(B_r)} \le C(n,p)\left(\bint_{B_{2r} }(\ez+|Xu^\ez|^2)^{\frac{p}{2}} \,dx\right)^{\frac1p}.
\end{equation}
Moreover, $ u^\ez\to u$ in $C^0(\overline U)$.
 \end{thm}
%\end{rem}
%
%
%
%
%
%
%
%
%
%
%%
%
%%
%%
%%
%%
%
%
%
%

\section{A pointwise upper bound for $(\Delta_0 u^\ez)^2$ via $|D^2_0u^\ez|^2-(\Delta_0u^\ez)^2$
} \label{alg}

In this section we prove  the following pointwise inequality.
 \begin{lem}\label{keylem1} For $1<p<\fz$ and
 $0<\ez<1$, let  $u^\ez\in C^\fz(U)$ be a solution to  \eqref{aplap}.
Then  \begin{align}\label{pi-lpap}
\left [\frac{n}{(p-2)^2}+\frac{1}{p-2}-(n-1)\right] (\Delta_0 u^\ez)^2 \le \frac{2n-1}{2} \left[|D_0^2u^\ez|^2-(\Delta_0 u^\ez)^2 \right]\quad \mbox{in $U$}.
\end{align}
\end{lem}
Above when $p = 2$, we make the convention that $\frac10 = \fz$ and $\fz \cdot 0 = 0$. Note that, when $p = 2$, \eqref{ap2} becomes
  $$0 \le \frac{2n-1}{2} |D_0^2u^\ez|^2,$$
 and hence \eqref{pi-lpap} reads as which holds trivially. To prove Lemma 3.1 with $p \ne 2$, we recall the following fundamental inequality
obtained in \cite[Lemma 2.1]{dpzz}.

\begin{lem}\label{fund} For any vectors $\vec\lz=(\lz_1,\cdots,\lz_{2n})$ and $\vec a=(a_1,\cdots,a_{2n})$
with $ |\vec a|=1$ we have
\begin{align*}
  & \left| \sum_{i=1}^{2n} \lz_i^2a_i^2 -   (\sum_{i=1}^{2n} \lz_i  )  (\sum_{j=1}^{2n} \lz_ja_j^2 ) -\frac12 [\sum_{i=1}^{2n} \lz_i^2-(\sum_{i=1}^{2n} \lz_i)^2] \right|\le
 (n-1)  [\sum_{i=1}^{2n} \lz_i^2-  \sum_{i=1}^{2n} \lz_i^2a_i^2  ].
 \end{align*}
\end{lem}

As a consequence of   Lemma \ref{fund} we have a pointwise inequality
for the algebraic structure of $\Delta_0 v\Delta_{0,\fz}v$ of $v\in C^\fz$, where and below
 $\Delta_{0,\fz}  $ is  the $\fz$-subLaplacian, that is,  $$\Delta_{0,\fz} v= \sum_{i,j=1}^{2n} X_ivX_iX_jvX_jv=(Xv)^TXXvXv=(Xv)^TD_0^2vXv.$$
  \begin{lem}\label{keylem}
  For any  $v\in C^\fz(U)$, we have
 \begin{align*}
 &\left|  |D_0^2vXv|^2  -    \dzz v \Delta_{0,\fz} v   - \frac12[ |D_0^2v|^2-(\dzz v)^2]|Xv|^2\right|\\
 &\quad\quad\quad\quad\quad\quad\quad\quad\quad
 \le (n-1) \left[|D_0^2v|{|Xv|^2}-  |D_0^2vXv|^2 \right] \ \mbox{in $U$}.
 \end{align*}
\end{lem}

\begin{proof}
 Let  $\bar x\in U $.
If $Xv(\bar x)=0$,   then the above inequality obviously holds.
We   assume that $Xv(\bar x)\ne 0$ below.  By  dividing both sides by $|Xv(\bar x)|^2$, we further assume that $|Xv(\bar x)|=1$.

At  $\bar x$, $D_0^2v$ is a symmetric matrix  and hence
 its eigenvalues are given by  $\{\lz_i\}_{i=1}^{2n}\subset\rr$.
 One may find an orthogonal matrix $O\in {\bf O}(2n)$ so that
$$O^TD_0^2vO={\rm diag}\{\lz_1,\lz_2,\ldots,\lz_{2n}\}.$$
Note that $O^{-1}=O^T$.
At $\bar x$, it follows that
$$
|D_0^2v|^2=|O^TD_0^2v O|^2=\sum_{i=1}^{2n}\lz_i^2,\quad \Delta_0 v=\sum_{i=1}^{2n} \lz_i.
$$
Writing $O^TXv=\sum_{i=1} a_i {\bf e}_i=:\vec{a}$,   we have
$$\Delta_{0,\fz} v=(Xv)^TD_0^2vXv=(O^T Xv)^T (O^T D_0^2v O) (O^TXv)=\sum_{i=1}^{2n} \lz_ia_i^2  $$
and $$|D_0^2vXv|^2= |(O^T D_0^2v O)(O^TXv)|^2=\sum_{i=1}^{2n} \lz_i^2a_i^2.$$
Applying Lemma \ref{fund} to $\vec\lz:=(\lz_1,\lz_2,\ldots,\lz_{2n}) $ and
$\vec a $, we obtain
 \begin{align*} &\left|  { |D_0^2vXv|^2} -  {\Delta_0 v \Delta_{0,\fz} v } -\frac12[|D_0^2v|^2-(\Delta_0 v)^2]|Xv|^2\right|\\
  &\quad =\left| \sum_{i=1}^{2n} \lz_i^2a_i^2 -   (\sum_{i=1}^{2n} \lz_i )  (\sum_{j=1}^{2n} \lz_ja_j^2 ) -\frac12[\sum_{i=1}^{2n} \lz_i^2-(\sum_{i=1}^{2n} \lz_i)^2] \right| \\
 &\quad\le
 (n-1) [\sum_{i=1}^{2n} \lz_i^2-  \sum_{i=1}^{2n} \lz_i^2a_i^2 ]\\
 &\quad= (n-1) \left[|D_0^2v|^2{|Xv|^2}-  |D_0^2vXv|^2 \right]
 \end{align*}
 as desired.
 \end{proof}

%Let $u$ be a $p$-harmonic function in $\Omega$.
%Given any   smooth domain $U\Subset \Omega$,   for $\ez\in(0,1]$
%we let $u^\ez\in W^{1,p}(U)\cap C^0(\overline U)$  be a weak solution to   the regularized   equation \eqref{aplap}. Thanks to Chapter 4 in \cite{dr}, we know that $u^\ez\in C^\fz(U )\cap C^0(\overline U )$ for any $p>1$, hence
% we can apply the algebraic inequality \eqref{keyineq} to $u^\ez$ and claim Lemma \ref{keylem1}
%i.e.,
%\begin{equation}\label{aplap}
%\sum_{i=1}^{2n}X_i[(\ez+|Xv|^2)^{\frac{p-2}{2}}X_iv]=0.  %\mbox{in $\Omega \subset \hn$}.
% \end{equation}
%Similarly, a weak solution to \eqref{aplap} is a function $v$ satisfying
%$$\sum_{i=1}^{2n}(\ez+|Xv|^2)^{\frac{p-2}{2}}X_ivX_i\phi =0  \quad \forall \phi\in C^\fz_c(\Omega).$$
%\begin{lem} \label{nonde}
%If $u^\ez\in C^\fz(U )\cap C^0(\overline U )$ is a solution to \eqref{aplap}, then
%  \begin{align}\label{alden1}
%[\frac{n}{(p-2)^2}+\frac{1}{p-2}-(n-1)] (\Delta_0 u^\ez)^2 \le \frac{2n-1}{2} [|D_0^2u^\ez|^2-(\Delta_0 u^\ez)^2].
%\end{align}
%\end{lem}

%
%Notice that since the solution $u^\ez$ to \eqref{aplap} is in $C^\fz(U)$,
%then \eqref{aplap} can be rewrite as
%\begin{equation}\label{ap2}
%(p-2)\Delta_{0,\fz}u^\ez+(\ez+|Xu^\ez|^2)\dzz u^\ez=0.
%\end{equation}
%Equation \eqref{ap2} will be used heavily in the remaining part.

 Now we use Lemma \ref{keylem} to prove Lemma \ref{keylem1}.

\begin{proof}[Proof of Lemma \ref{keylem1}]
Since   $u^\ez\in C^\fz(U)$,
  equation \eqref{aplap} gives
\begin{equation}\label{ap2}
(p-2)\Delta_{0,\fz}u^\ez+(\ez+|Xu^\ez|^2)\dzz u^\ez=0\quad \mbox{in $U$}.
\end{equation}
First, if $p=2$, \eqref{ap2} becomes
$$ \dzz  u^\ez =0 \quad \mbox{in $U$}.$$
In addition, one has
$$   0 \le \frac{2n-1}{2} |D_0^2u^\ez|^2, $$
holds trivially. Hence, via the convention
$\frac1{p-2}= \fz$ and $\fz\cdot 0 = 0$, \eqref{pi-lpap} holds.

Below we assume $p \ne 2$.
Let $\bar x \in U$ be any point. If $Xu^\ez(\bar x )=0$, then by \eqref{ap2} we may always have
$\Delta_0 u^\ez(\bar x)=0$. Thus \eqref{pi-lpap} holds.
  Below we assume  $Xu^\ez(\bar x )\ne 0$.
 Applying Lemma 4.1 to $u^\ez$,  by \eqref{ap2}, at $\bar x $ one gets
 \begin{align*}
&{ |D_0^2u^\ez Xu^\ez |^2} + \frac{(\Delta_0 u^\ez )^2}{p-2} [|Xu^\ez|^2+\ez]-\frac12[|D_0^2u^\ez|^2-(\Delta_0 u^\ez)^2]|Xu^\ez|^2 \\
&\quad \le  (n-1) [|D_0^2u^\ez|^2{|Xu^\ez |^2}-  |D_0^2u^\ez Xu^\ez|^2 ].
 \end{align*}
Dividing both sides by $|Xu^\ez(\bar x )|^2$, at $\bar x $ we obtain
\begin{equation}\label{e3.z1} n \frac{ |D_0^2u^\ez Xu^\ez |^2}{|Xu^\ez|^2} +  \frac1{p-2}\frac{(\Delta_0 u^\ez )^2} {|Xu^\ez|^2} [|Xu^\ez|^2+\ez]   \le \frac12[|D_0^2u^\ez|^2-(\Delta_0 u^\ez)^2]   +(n-1) |D_0^2u^\ez|^2.
\end{equation}
Using \eqref{ap2} again  and Young's inequality, at $\bar x$ one has
 $$\frac{ |D_0^2u^\ez Du^\ez |^2}{|Xu^\ez|^2}\ge
\frac{ |\Delta_{0,\fz} u^\ez |^2}{|Xu^\ez|^4}\ge  \frac1{(p-2)^2}\frac{(\Delta_0 u^\ez)^2}{|Xu^\ez|^2} [|Xu^\ez|^2+\ez].$$
This and \eqref{e3.z1} yield
\begin{align}
                    \label{eq7.01}
&[\frac{n}{(p-2)^2}+\frac{1}{p-2}] \left(\frac{\Delta_0 u^\ez}{|Xu^\ez|}\right)^2[|Xu^\ez|^2+\ez] \le\frac{1}{2} [|D_0^2u^\ez|^2-(\Delta_0 u^\ez)^2] +(n-1) |D_0^2u^\ez|^2.
\end{align}
Since
 $$\frac{n}{(p-2)^2}+\frac{1}{p-2}>0,$$
Inequality \eqref{eq7.01} further gives
\begin{align*}
[\frac{n}{(p-2)^2}+\frac{1}{p-2}] (\Delta_0 u^\ez)^2 \le\frac{1}{2} [|D_0^2u^\ez|^2-(\Delta_0 u^\ez)^2] +(n-1) |D_0^2u^\ez|^2.
\end{align*}
Subtracting $(n-1)(\Delta_0 u^\ez)^2$ on both sides, we obtain
$$[\frac{n}{(p-2)^2}+\frac{1}{p-2}-(n-1)] (\Delta_0 u^\ez)^2 \le \frac{2n-1}{2} [|D_0^2u^\ez|^2-(\Delta_0 u^\ez)^2],$$
that is, \eqref{pi-lpap}.
\end{proof}

\section{An integral upper bound for   $|D_0^2u^\ez|^2-(\Delta_0 u^\ez)^2$}  \label{divv}

In this section we prove the following.
\begin{lem} \label{keylem2}
 For $2<p\le 4$ and
 $0<\ez<1$, let  $u^\ez\in C^\fz(U)$ be a solution to  \eqref{aplap}.
Then, for any $\eta>0$ and $\phi \in C^{\fz}_{c}(U)$,
\begin{align} \label{key2}
         &\bigg | \int_U  \big [|D_0^2u^\ez|^2  -(\Delta_0 u^\ez)^2 \big] \phi^2 \,dx \bigg | \nonumber\\
         & \le  \eta\int_U|XXu^\ez|^2\phi^2\,dx +C(n,p,\eta)K_\phi \left[  \int_{\rm spt(\phi)} |Xu^\ez|^{ {4-p} }   \,dx\right]^{1/2}\nonumber\\
     &\quad\quad\times \left[  \int_{\rm spt\,(\phi)} (\ez+|Xu^\ez|^2)^{\frac{p+2}{2}} \,dx  \int_{\rm spt(\phi)} (\ez+|Xu^\ez|^2)^{\frac{p-2}{2}} \phi^4 \,dx\right ]^{1/4}.
    \end{align}
\end{lem}
To prove this, firstly we establish the following identity.
\begin{lem} \label{lem3.1}
  For any $v\in C^\fz(U)$ and $\phi \in C^{\fz}_{c}(U)$, we have
  \begin{align} \label{div}
  \int_U \big [|D_0^2v|^2  -(\Delta_0 v)^2 \big] \phi^2 \,dx & =
\sum_{i,j=1}^n\int_U (X_iX_iv)(X_jv) X_j\phi^2 \,dx-  \sum_{i,j=1}^n\int_U (X_iX_jv)(X_jv) X_i\phi^2\,dx \notag \\
  & +\frac32\sum_{i =1}^n \int_U X_{i }vX_{i+n}Tv \phi^2 \,dx - \frac32\sum_{i =1}^n\int_U X_{i+n}vX_iTv \phi^2 \,dx.
\end{align}
\end{lem}

\begin{proof}
% Recall that
%  $$D_0^2v= \left (\frac{X_iX_{j}v+X_{j}X_iv}2
%            \right )_{1\le i,j\le 2n}  .$$
Note that
\begin{align*}
   |D_0^2v|^2    &= \sum_{1\le i,j\le 2n,|i-j|\ne n}    (X_iX_{j}v) (X_jX_{i}v) \\
   &\quad\quad \quad  +
   \frac12\sum_{i=1}^{n} [ (X_iX_{i+n}v)^2+ 2(X_i X_{i+n} v)(X_{i+n}X_i v)+(X_{i+n}X_iv)^2]
\end{align*}
and that
\begin{align*}
   |\Delta_0 v|^2    &= \sum_{1\le i,j\le 2n,|i-j|\ne n}    (X_iX_{i}v) (X_jX_{j}v)  +
   2\sum_{i=1}^{n}   (X_iX_{i}v)(X_{i+n}X_{i+n} v).
\end{align*}
The proof of \eqref{div} is then reduced to
 prove that
for $1\le i,j\le 2n$ with $|i-j|\ne n$,  we have
   \begin{align}\label{claim1}
   &\int_U\big[(X_iX_jv)(X_iX_jv)-(X_iX_iv)(X_jX_jv)\big]\phi^2\,dx\nonumber\\
   &\quad=   \int_U (X_iX_iv)(X_jv) X_j\phi^2 \,dx-  \int_U (X_iX_jv)(X_jv) X_i\phi^2\,dx
    \end{align}
    and that  for $1\le i\le n$,  we have
\begin{align} \label{claim2}&\int_U [(X_{i+n}X_iv)^2+(X_iX_{i+n}v)^2] \phi^2 \,dx-2\int_UX_iX_iv X_{i+n}X_{i+n}v \phi^2 \,dx \,dx\nonumber\\
 &\quad = -\int_U X_{i+n}v X_iX_{i+n}v X_i\phi^2 \,dx+  \int_UX_{i+n}vX_iX_iv  X_{i+n}\phi^2 \,dx - 2\int_U X_{i+n}vX_iTv \phi^2 \,dx\nonumber\\
  &\quad\quad-\int_U X_{i }v X_{i+n}X_iv X_{i+n}\phi^2 \,dx+  \int_UX_{i}vX_{i+n}X_{i+n}v  X_{i }\phi^2 \,dx + 2\int_U X_{i }vX_{i+n}Tv \phi^2 \,dx
  \end{align}
  and also
   \begin{align}\label{claim3}&\int_U 2 X_iX_{i+n}v X_{i+n}X_iv\phi^2 \,dx-2\int_U X_{i+n}X_{i+n}vX_iX_iv \phi^2 \,dx \nonumber
\\
  &\quad=   - \int_U X_{i+n}vX_{i+n}X_iv X_i\phi^2 \,dx
  +\int_U X_{i+n}X_{i+n}vX_iv X_i\phi^2 \,dx -\int_U X_{i+n}TvX_iv \phi^2 \,dx\nonumber\\
    &\quad\quad   - \int_U X_{i }v X_iX_{i+n}v X_{i+n}\phi^2 \,dx
  +\int_U X_{i }X_{i }vX_{i+n}v X_{i+n}\phi^2 \,dx +\int_U X_{i }TvX_{i+n}v \phi^2 \,dx.
  \end{align}

To see  \eqref{claim1}, since $|i-j|\ne n$, by integration by parts twice and $X_iX_jv=X_jX_iv$, we have
 \begin{align*}
     &   \int_U   (X_iX_jv)(X_iX_jv)  \phi^2\,dx  \\
              & \quad= -\int_U (X_iX_iX_jv)(X_jv) \phi^2 \,dx -  \int_U (X_iX_jv)(X_jv) X_i\phi^2 \,dx \\
              & \quad= \int_U (X_iX_iv)(X_jX_jv) \phi^2 \,dx+ \int_U (X_iX_iv)(X_jv) X_j\phi^2 \,dx-  \int_U (X_iX_jv)(X_jv) X_i\phi^2 \,dx.
    \end{align*}
    That is
    \begin{align*}
      \int_U\big[(X_iX_jv)(X_iX_jv)-&(X_iX_iv)(X_jX_jv)\big]\phi^2\,dx \\
       & =   \int_U (X_iX_iv)(X_jv) X_j\phi^2 \,dx-  \int_U (X_iX_jv)(X_jv) X_i\phi^2\,dx,
    \end{align*}
    which gives \eqref{claim1}.

To see \eqref{claim2}, by integration by parts,
     $$\int_U (X_iX_{i+n}v)^2\phi^2 \,dx   = - \int_U X_{i+n}vX_iX_iX_{i+n}v \phi^2 \,dx - \int_U X_{i+n}v X_iX_{i+n}v X_i\phi^2 \,dx.$$
     Noting that
    $$X_iX_iX_{i+n}v=X_iX_{i+n}X_iv+ X_iTv=X_{i+n}X_iX_iv+ 2X_iTv,$$
    by integration by parts again, we have
    \begin{align*}- \int_U X_{i+n}vX_iX_iX_{i+n}v \phi^2 \,dx&=
    - \int_U X_{i+n}vX_{i+n}X_iX_iv \phi^2 \,dx  - 2\int_U X_{i+n}vX_iTv \phi^2 \,dx\\
    &= \int_U X_{i+n}X_{i+n}vX_iX_iv \phi^2 \,dx
     - 2\int_U X_{i+n}vX_iTv \phi^2 \,dx\\
     &\quad+  \int_U X_{i+n}vX_iX_iv  X_{i+n}\phi^2 \,dx.
    \end{align*}
 Thus,
 \begin{align*} \int_U & (X_iX_{i+n}v)^2\phi^2 \,dx-\int_UX_iX_iv X_{i+n}X_{i+n}v \phi^2 \,dx\\
 &= -\int_U X_{i+n}v X_iX_{i+n}v X_i\phi^2 \,dx +\int_U X_{i+n}vX_iX_iv  X_{i+n}\phi^2 \,dx
  - 2\int_U X_{i+n}vX_iTv \phi^2 \,dx
 \end{align*}

 Moreover, by integration by parts,
     $$\int_U (X_{i+n}X_iv)^2\phi^2 \,dx   = - \int_U X_{i }vX_{i+n}X_{i+n}X_{i }v \phi^2 \,dx - \int_U X_{i}v X_{i+n}X_iv X_{i+n}\phi^2 \,dx.$$
     Noting that
    $$X_{i+n}X_{i+n}X_{i }v  =X_{i+n}X_{i }X_{i+n}v - X_{i+n}Tv=X_{i }X_{i+n}X_{i+n} v- 2X_{i+n}Tv,$$
    in a similar way we have
     \begin{align*} \int_U & (X_{i+n}X_iv)^2\phi^2 \,dx  - \int_UX_iX_iv X_{i+n}X_{i+n}v \phi^2 \,dx\\
 &=  -\int_U X_{i }v X_{i+n}X_iv X_{i+n}\phi^2 \,dx+  \int_U X_{i}vX_{i+n}X_{i+n}v  X_{i }\phi^2 \,dx  +2\int_U X_ivX_{i+n}Tv \phi^2 \,dx.
 \end{align*}
 Combining these together we get \eqref{claim2}.

 To see \eqref{claim3}, by integration by parts, we have

  $$\int_U X_iX_{i+n}v X_{i+n}X_iv\phi^2 \,dx
  = - \int_U X_{i+n}vX_iX_{i+n}X_iv \phi^2 \,dx - \int_U X_{i+n}vX_{i+n}X_iv X_i\phi^2 \,dx.$$
 Since  $$X_{i+n}vX_iX_{i+n}X_iv= X_iX_{i+n}vX_{i+n}X_iv-X_{i+n}Tv,$$  by integration by parts,
 \begin{align*}- \int_U X_{i+n}vX_iX_{i+n}X_iv \phi^2 \,dx&= - \int_U X_iX_{i+n}vX_{i+n}X_iv \phi^2 \,dx-
  \int_U X_{i+n }TvX_iv \phi^2 \,dx\\
  &= \int_U X_{i+n}X_{i+n}vX_iX_iv \phi^2 \,dx-
  \int_U X_{i+n }TvX_iv \phi^2 \,dx\\
  &\quad+\int_U X_{i+n}X_{i+n}vX_iv X_i\phi^2 \,dx.
  \end{align*}
  Thus
 \begin{align*}&\int_U X_iX_{i+n}v X_{i+n}X_iv\phi^2 \,dx-\int_U X_{i+n}X_{i+n}vX_iX_iv \phi^2 \,dx
\\
  &\quad=    -\int_U X_{i+n}vX_{i+n}X_iv X_i\phi^2 \,dx
  +\int_U X_{i+n}X_{i+n}vX_iv X_i\phi^2 \,dx- \int_U X_{i+n }TvX_iv \phi^2 \,dx.
  \end{align*}
  In a similar way, we have
 \begin{align*}&\int_U X_{i+n}X_iv X_iX_{i+n}v \phi^2 \,dx-\int_U X_{i+n}X_{i+n}vX_iX_iv \phi^2 \,dx
\\
  &\quad=   - \int_U X_{i }v X_iX_{i+n}v X_{i+n}\phi^2 \,dx
  +\int_U X_{i }X_{i }vX_{i+n}v X_{i+n}\phi^2 \,dx +\int_U X_{i }TvX_{i+n}v \phi^2 \,dx.
  \end{align*}
  Combining these together we have \eqref{claim3}.
\end{proof}

Next as a consequence of Lemmas \ref{xt} and  \ref{tu4} we have the following result.

\begin{lem}\label{lem3.2}
  Let $2<p\le 4$,
 $0<\ez<1$ and let  $u^\ez\in C^\fz(U)$ be a solution to  \eqref{aplap}.
For  any $\phi \in C^\fz_c(U)$, we have
  \begin{align}\label{yy1}
   \int_U |Xu^\ez||XTu^\ez| \phi^2 \,dx
     &\le  C(n,p)K_\phi\left[  \int_{\rm spt\,(\phi)} |Xu^\ez|^{ {4-p} }   \,dx\right]^{1/2}\nonumber\\
     &\quad\quad\times
      \left[  \int_{\rm spt\,(\phi)} (\ez+|Xu^\ez|^2)^{\frac{p+2}{2}} \,dx \int_{\rm spt\,(\phi)} (\ez+|Xu^\ez|^2)^{\frac{p-2}{2}} \,dx\right ]^{1/4}.
  \end{align}
\end{lem}

\begin{proof}
Since  $2<p \le 4$, by H\"older's inequality we have
 \begin{align*}
  \int_U |Xu^\ez||XTu^\ez| \phi^2 \,dx\le
  \left[\int_U   (\ez+|Xu^\ez|^2)^ {\frac  {p-2}2  }|XTu^\ez|^2 \phi^4 \,dx\right]^{1/2}
  \left[  \int_{\rm spt(\phi)} |Xu^\ez|^{ {4-p} }  \,dx\right]^{1/2}
   \end{align*}
Applying Lemma \ref{xt} with $\bz=0$ therein, we obtain
$$\int_U   (\ez+|Xu^\ez|^2)^ {\frac  {p-2}2  }|XTu^\ez|^2 \phi^4 \,dx\le C(n,p) K_\phi\int_U (\ez+|Xu^\ez|^2)^{\frac{p-2}{2}}|Tu^\ez|^2  \phi^2 \,dx.$$
By H\"older's inequality again we have
\begin{align*}\int_U &  (\ez+|Xu^\ez|^2)^ {\frac  {p-2}2  }|XTu^\ez|^2 \phi^4 \,dx\\
&\le C(n,p) K_\phi
\left[\int_U (\ez+|Xu^\ez|^2)^{\frac{p-2}{2}}|Tu^\ez|^4 \phi^4 \,dx\right]^{1/2}\left[  \int_{\rm spt(\phi)} (\ez+|Xu^\ez|^2)^{\frac{p-2}{2}}  \,dx\right]^{1/2}.
  \end{align*}
By Lemma \ref{tu4} with $\bz=2$ therein,  one has
\begin{align*}\int_U (\ez+|Xu^\ez|^2)^{\frac{p-2}{2}}|Tu^\ez|^4 \phi^4 \,dx\le  C(n,p)
 K_\phi^2
 \int_{\rm spt(\phi)} (\ez+|Xu^\ez|^2)^{\frac{p+2}{2}}  \,dx.
%
%\int_U &  (\ez+|Xu^\ez|^2)^ {\frac  {p-2}2  }|XTu^\ez|^2 \phi^4 \,dx\\
%&\le C(n,p) K_\phi^2
%\left[\int_{\rm spt(\phi)} (\ez+|Xu^\ez|^2)^{\frac{p+2}{2}}  \,dx\right]^{1/2}\left[  \int_{\rm spt(\phi)} (\ez+|Xu^\ez|^2)^{\frac{p-2}{2}}   \,dx\right]^{1/2}.
  \end{align*}
Form these we conclude  \eqref{yy1}.
%
%  \begin{align*}
%    &\int_U |Xu^\ez||XTu^\ez| \phi^6 \,dx\\
%     &\quad \le C(n,p)   \left[\int_U (\ez+|Xu^\ez|^2)^{\frac{p-2}{2}}|Tu^\ez|^2 |X\phi|^2\phi^4 \,dx\right]^{1/2}\left[  \int_U |Xu^\ez|^{ {4-p} }  \phi^6 \,dx\right]^{1/2}\\
%     &\quad\le  \left[\int_U (\ez+|Xu^\ez|^2)^{\frac{p-2}{2}}|Tu^\ez|^4  \phi^6 \,dx\right]^{1/4}\\
%     &\quad\quad\times
%        \left[\int_U (\ez+|Xu^\ez|^2)^{\frac{p-2}{2}}|X\phi|^4\phi^4 \,dx\right ]^{1/4}\left[  \int_U |Xu^\ez|^{ {4-p} }  \phi^6 \,dx\right]^{1/2}.
%  \end{align*}
 % we conclude \eqref{yy1}.
  %\begin{align*}
%   \int_U |Xu^\ez||XTu^\ez| \phi^6 \,dx
%     &\le  C(n,p)K_\phi  \left[  \int_{\rm spt\,(\phi)} (\ez+|Xu^\ez|^2)^{\frac{p+2}{2}} \,dx \right]^{1/4}\\
%     &\quad\times
%        \left[\int_U (\ez+|Xu^\ez|^2)^{\frac{p-2}{2}} \phi^4 \,dx\right )^{1/4}\left[  \int_U |Xu^\ez|^{ {4-p} }  \phi^6 \,dx\right]^{1/2}.
%  \end{align*}
\end{proof}

Finally, Lemma \ref{keylem2} follows from Lemmas \ref{lem3.1} and \ref{lem3.2}.
\begin{proof}[Proof of Lemma \ref{keylem2}]
Applying \eqref{div} to $u^\ez$ and   Young's inequality, we have
\begin{align*}
  \int_U \big [|D_0^2 & u^\ez|^2  -(\Delta_0 u^\ez)^2 \big] \phi^2 \,dx \notag \\
  & =\int_U (X_iX_iu^\ez)(X_ju^\ez) X_j\phi^2 \,dx-  \int_U (X_iX_ju^\ez)(X_ju^\ez) X_i\phi^2\,dx \notag \\
  & \quad + \frac{3n}2 \int_U X_{i }u^\ez X_{i+n}Tu^\ez \phi^2 \,dx - \frac{3n}2\int_U X_{i+n}u^\ez X_iTu^\ez \phi^2 \,dx \notag \\
  & \le \eta\int_U|XXu^\ez|^2\phi^2\,dx+ C(n,\eta)\int_U |Xu^\ez|^2|X\phi|^2 \,dx  + 3n \int_U |Xu^\ez| |XTu^\ez| \phi^2 \,dx.
\end{align*}
Note that, by Lemma \ref{lem3.2},   the last term is bounded by the right hand side of \eqref{yy1}, and also that, by H\"{o}lder's inequality, the second term is also bounded by the right hand side of \eqref{yy1},
 that is,
 \begin{align}\label{yy5}
 \int_U |Xu^\ez|^2  |X\phi|^2\,dx &\le K_\phi\left[  \int_{\rm spt\,(\phi)} |Xu^\ez|^{ {4-p} }   \,dx\right]^{1/2}\nonumber\\
     &\quad\quad\times
      \left[  \int_{\rm spt\,(\phi)} (\ez+|Xu^\ez|^2)^{\frac{p+2}{2}} \,dx \int_{\rm spt\,(\phi)} (\ez+|Xu^\ez|^2)^{\frac{p-2}{2}} \,dx\right ]^{1/4}.
 \end{align}
 The proof is complete.
\end{proof}

\section{Proofs of Theorem \ref{ez} and Theorem 1.1}  \label{main}

Here, we prove Theorems \ref{ez} and \ref{thm1} in order.

\begin{proof} [Proof of Theorem \ref{ez}]
  If  $1<p  \le 2$,  applying Lemma \ref{xxu} with $\bz=2-p>0$  one has
    \begin{align}\label{yy4}
   \int_U  |XXu^\ez|^2 & \phi^2 \,dx  \le   C(n,p) K_{\phi }\int_{\rm spt\,(\phi)} (\ez+|Xu^\ez|^2) \,dx+
    C(n,p)\int_U |Tu^\ez|^2\phi^2\, dx.
    \end{align}
By H\"older's inequality and Lemma \ref{tu4} with $\bz=2$ therein, we have
  \begin{align*}
    \int_U |Tu^\ez|^2\phi^2\, dx
    & \le \left[\int_{\rm spt\,(\phi)} (\ez+ |Xu^\ez|^2)^{\frac{2-p}{2}}  \,dx \right]^{1/2} \left[\int_U (\ez+ |Xu^\ez|^2)^{\frac{p-2}{2}}|Tu^\ez|^4\phi^4 \,dx \right]^{1/2}\\
     & \le C(n,p)K_\phi \left[\int_{\rm spt\,(\phi)}  (\ez+ |Xu^\ez|^2)^{\frac{2-p}{2}} \,dx \right]^{1/2}  \left[\int_{\rm spt\,(\phi)}   (\ez+|Xu^\ez|^2)^{\frac{2+p}{2}}\,dx\right]^{1/2}.
  \end{align*}
 Applying H\"older's inequality to the first term in the right hand side of \eqref{yy4} we get \eqref{bdd}.

  Below we assume that $2<p\le 4$.
Integrating   \eqref{pi-lpap} in Lemma \ref{keylem1}, for any $\phi\in C_c^\fz(U)$ we obtain
\begin{align}\label{yy2}&\left [\frac{n}{(p-2)^2}+\frac{1}{p-2}-(n-1) \right ]\int_U (\Delta_0 u^\ez)^2 \phi^2 \,dx \le \frac{2n-1}{2} \int _U [|D_0^2u^\ez|^2-(\Delta_0 u^\ez)^2]\phi^2 \,dx
\end{align}
Note that \eqref{rangep} gives $$ \frac{n}{(p-2)^2}+\frac{1}{p-2}-(n-1)>0.$$
Applying  Corollary \ref{cor2.2} to the left hand side of \eqref{yy2} and  Lemma \ref{keylem2} to the right hand side of \eqref{yy2} one has
\begin{align*}
    &\frac13 \left [\frac{n}{(p-2)^2}+\frac{1}{p-2} -(n-1)\right] \int_U |XX u^\ez|^2 \phi^2\, dx \\
  & \quad \le \eta  \int_U |XX u^\ez|^2\phi^2 \,dx +C(n,p) \int_U \left[|Xu^\ez|^2|X\phi |^2+|u^\ez|^2|XX\phi|^2 \right] \,dx\nonumber\\
  &\quad\quad+ C(n,p,\eta)K_\phi\left[  \int_{\rm spt\,(\phi)}  |Xu^\ez|^{ {4-p} }   \,dx\right]^{1/2}\\
        &\quad\quad\quad \times \left[  \int_{\rm spt\,(\phi)} (\ez+|Xu^\ez|^2)^{\frac{p+2}{2}} \,dx \int_{\rm spt\,(\phi)} (\ez+|Xu^\ez|^2)^{\frac{p-2}{2}}  \,dx\right ]^{1/4}.
\end{align*}
Taking   $\eta>0$ sufficiently small, and using \eqref{yy5}, we obtain \eqref{yy3}.
\end{proof}

With the help of Theorem \ref{ez}, we are ready to  show Theorem \ref{thm1}.
\begin{proof}[Proof of Theorem \ref{thm1}]
Let $u $ be any  $p$-harmonic function  in $\Omega\subset\hh^n$.
Given  any smooth domain $U\Subset\Omega$, for $\ez\in(0,1]$ let $u^\ez\in C^\fz(U)$
be a solution to \eqref{aplap}.
We first show that $u^\ez \to u$ weakly in $HW^{2,2}_\loc(U)$. By Theorem \ref{unif}, we have
 $$u^\ez\to u \mbox{ in $C^{0}(\overline U)$ as $\ez \to 0$},$$
 $$ Xu \in L^\fz_{\loc}(U, \rr^{2n}),$$
 $$ Xu^\ez \in L^\fz_{\loc}(U, \rr^{2n})$$
 uniformly in $\ez > 0$. Consequently, $u \in L^\fz(U)$ and $u^\ez \in  L^\fz(U)$ uniformly in $\ez > 0$. Using this and choosing
suitable test functions $\phi \in C^\fz_c (U)$ in \eqref{bdd} and \eqref{yy3} as given in Theorem \ref{ez}, we conclude $u^\ez \in HW^{2,2}_{\loc}(U)$,
uniformly in $\ez \in (0, 1]$.

 \smallskip
Next, we claim that $u \in HW^{2,2}_\loc(U)$, $Xu^\ez\to Xu$   in $L^{q}_\loc(U,\rr^{2n})$ as $\ez \to 0$ and $XXu^\ez \to XXu$ weakly in $L^{2}_\loc(U,\rr^{2n}\times \rr^{2n})$ as $\ez \to 0$.
  To see this, for any subdomain $V \Subset U$, we already have
  $$   \sup_{\ez \in (0,1]} \|XXu^\ez\|_{L^{2}(V,\rr^{2n}\times \rr^{2n})} <\fz.$$

For any subsequence $\{\ez_j\}_j \in \nn$ which converges to $0$, by the weak compactness of $HW^{2,2}(V)$, up to some subsequence one has $XXu^{\ez_j} \to XXv$ weakly in $L^{2}(V,\rr^{2n}\times \rr^{2n})$ for some function $v \in HW^{2,2}(V)$. Since $u^\ez  \to u$ in
$C^0(U)$, we have $XXu^\ez \to XXu|_U$ in the distributional sense. Thus, $XXv = XXu|_V$ in the distributional sense.
We therefore have $u \in HW^{2,2}(V)$. By the arbitrariness of the subsequence $\{\ez_j\}$, we have $XXu^\ez \to  XXu$ weakly
in $L^{2}(V,\rr^{2n}\times \rr^{2n})$ as $\ez \to 0$. Using this and H\"{o}lder's inequality, noting $Xu \in L^\fz(V, \rr^{2n})$ and $Xu^\ez \in L^\fz(V, \rr^{2n})$
uniformly in $\ez > 0$, we further conclude $Xu^\ez \to Xu$ in $L^q(V, \rr^{2n})$ for any $0 < q < \fz$ as $\ez \to 0$; here we omit
the details. By the arbitrariness of $V \Subset U$, we get the desired claim.

   Finally, considering the above claim and letting $\ez \to 0$ in \eqref{bdd} and \eqref{yy3}, in a standard way we conclude \eqref{bdd0} and \eqref{yy30} respectively. The proof is complete.
\end{proof}

\end{document}